\documentclass[11pt,reqno]{amsart}
\usepackage{xifthen,setspace}
\usepackage{bbm}
\usepackage{pkgfile}
\newtheorem{example}{Example}

\spacing{1.025}
\allowdisplaybreaks

\title[Replica Symmetry in Upper Tails of Mean-Field Hypergraphs]{Replica Symmetry in Upper Tails of Mean-Field Hypergraphs}

\author[Mukherjee]{Somabha Mukherjee}
\address{Department of Statistics, University of Pennsylvania, Philadelphia, USA,
{\tt somabha@wharton.upenn.edu}}

\author[Bhattacharya]{Bhaswar B. Bhattacharya}
\address{Department of Statistics, University of Pennsylvania, Philadelphia, USA,
{\tt bhaswar@wharton.upenn.edu}}

\begin{document}


\maketitle

\begin{abstract} Given a sequence of $s$-uniform hypergraphs $\{H_n\}_{n \geq 1}$, denote by $T_p(H_n)$ the number of edges in the random induced hypergraph obtained by including every vertex in $H_n$ independently with probability $p \in (0, 1)$. Recent advances in the large deviations of low complexity non-linear functions of independent Bernoulli variables can be used to show that tail probabilities of  $T_p(H_n)$ are precisely approximated by the so-called `mean-field' variational problem, under certain assumptions on the sequence $\{H_n\}_{n \geq 1}$. In this paper, we study properties of this variational problem for the upper tail of  $T_p(H_n)$, assuming that the mean-field approximation holds.  In particular, we show that the variational problem has a universal {\it replica symmetric} phase (where it is uniquely minimized by a constant function), for any sequence of {\it regular} $s$-uniform hypergraphs, which depends only on $s$. We also analyze the associated variational problem for the related problem of estimating subgraph frequencies in a converging sequence of dense graphs. Here, the variational problems themselves have a limit which can be expressed in terms of the limiting graphon. 
\end{abstract}


\section{Introduction}\label{introduction}

Given a $s$-uniform hypergraph $H=(V(H),E(H))$ with vertex set $V(H)$ and hyperedge set $E(H)$ (which is a collection of $s$-element subsets of $V(H)$) and $p \in (0, 1)$ fixed, construct a random sub-hypergraph  of $H$ as follows: sample each vertex in $V(H)$ with probability  $p$ and consider the induced sub-hypergraph of $H$ on the set of sampled vertices. Denote by $T_p(H)$ the number of edges in this random sub-hypergraph, which 
can be written, more formally, as 
\begin{align}\label{eq:TpH}
T_p(H) := \sum_{e \in E(H)} \prod_{v \in e} X_v, 
\end{align}
where $X_1, X_2, \ldots,X_{|V(H)|}$ are i.i.d. $\mathrm{Ber}(p)$. (Note that $\E(T_p(H))= |E(H)| p^s$.) Numerous celebrated problems in combinatorial probability can be re-formulated in terms of \eqref{eq:TpH}, for some choice of $H$. 

\begin{itemize}

\item[(1)] {\it Subgraphs in Erd\H os-R\'enyi random graphs}: The number of copies of a fixed graph $F=(V(F), E(F))$ in the Erd\H os-R\'enyi random graph $G(n, p)$ can be formulated in terms of \eqref{eq:TpH} as follows: Consider the hypergraph $H_n(F)$ (to be referred to as the {\it $F$-counting hypergraph}) with vertex set of $H_n(F)$ as the edge set of the complete graph $K_n$, and edge set of $H_n(F)$ as the collection of the edge sets of all copies of $F$ in $K_n$. Then $H_n(F)$ is a $|E(F)|$-uniform regular hypergraph\footnote{A hypergraph $H=(V(H), E(H))$ is said to be $d$-{\it regular} if every vertex $v \in V(H)$ is incident on exactly $d$ hyperedges in $E(H)$.} and $T_p(H_n(F))$ is precisely the number of copies of $F$ in $G(n, p)$. 

\item[(2)] {\it Arithmetic progressions in a random set}: Given $r\geq 1$, define the {\it $r$-AP counting hypergraph} $H_n(r)$ as the hypergraph with vertex set $[n]:=\{1, 2, \ldots, n\}$ and edge set the set of all $r$ term arithmetic progressions in $[n]$. Then $T_p(H_n(r))$ is the number of $r$-term arithmetic progressions in a random subset of $[n]$, where every element is included in the subset independently with probability $p$. 

\item[(3)] {\it Estimating motif frequencies in large graphs}: Efficiently counting motifs in a large graph such as the number of edges or triangles (more generally, subgraph counts) is an important statistical and computational problem \cite{KW}. One natural strategy to reduce storage and computational costs is to randomly sample subsets of vertices, where natural estimates of subgraph counts are often of the form \eqref{eq:TpH} above (see Section \ref{sec:dense} for details). 

\end{itemize}

Concentration inequalities for $T_p(H_n)$, for a sequence of hypergraphs $\{H_n\}_{n \geq 1}$,  are well-known (see \cite{JR11,War} and the references therein). In this paper, we study the precise large deviations upper tail probabilities of $T_p(H_n)$, in the {\it fixed} $p \in (0, 1)$ regime, which involves determining the exact asymptotics of 
\begin{align}\label{eq:P}
\log \P(T_p(H_n) \geq r^s |E(H_n)|), \quad \text{for} \quad 0 < p < r < 1.
\end{align}
Note that $T_p(H_n)$ is a random multi-linear polynomial indexed by $H_n$, and establishing its upper tail asymptotics falls in the framework of non-linear large deviations, introduced in the seminal paper of Chatterjee and Dembo \cite{CD16}. Here, they studied the large deviations of a general random function $f(X_1, X_2, \ldots, X_n)$, where $f: \{0, 1\}^n \rightarrow \R$ and $X_1, X_2, \ldots, X_n$ are i.i.d. $\dBer(p)$, and came up with a notion of complexity of the gradient of the function (along with additional smoothness properties), under which the tail probabilities of   $f(X_1, X_2, \ldots, X_n)$ (and the associated Gibbs measure on $\{0, 1\}^n$) can be well-approximated by the so-called `mean-field' variational problem, which is an entropic variational problem over the set of product measures on the hypercube $\{0, 1\}^n$. Thereafter, Eldan \cite{eldan} obtained an improved set of conditions, which involved computing the Gaussian width of the gradient of the function, under which the above reduction holds. Similar results for Gibbs measures beyond the hypercube were obtained by \cite{BM17,yan}, and recently by Austin \cite{austin} for very general product spaces.

These results can be used to obtain various sufficient conditions on a sequence of hypergraphs $\{H_n\}_{n\geq 1}$ for which the probability in \eqref{eq:P} can be approximated by the corresponding mean-field variational problem. This motivates the following abstract definition:

\begin{defn}\label{defn:approx}(Mean-field hypergraphs) Given a $s$-uniform hypergraph $H$ and $p \in (0, 1)$, the upper-tail mean field variational problem for $T_p(H)$ is defined as: 
\begin{equation}\label{varprob}
\phi_{H}(r, p) = \inf_{\bm x \in [0,1]^{|V(H)|}} \left\{\frac{1}{|V(H)|}\sum_{v=1}^{|V(H)|} I_p(x_v) :~t(H,\bm x) \geq r^s|E(H)|\right\},
\end{equation}
where
\begin{itemize}
\item[--] $0 < p < r < 1$, 
\item[--] $I_p(x) = x\log\frac{x}{p} + (1-x)\log\frac{1-x}{1-p}$, is the relative entropy of $\dBer(x)$ with respect to $\dBer(p)$; and

\item[--] $t(H,\bm x) = \sum_{e \in E(H)}\prod_{v \in e} x_v$, for $\bm x =\left(x_1,\ldots,x_{|V(H)|}\right) \in [0,1]^{|V(H)|}$. 

\end{itemize}
Moreover, a sequence $\{H_n\}_{n \geq 1}$ of $s$-uniform hypergraphs  is said to be {\it mean-field} (for the upper tail problem) if  
\begin{align}\label{eq:phi}
\lim_{n\rightarrow \infty} \frac{\frac{1}{|V(H_n)|}\log \mathbb{P}(T_p(H_n)\geq r^s |E(H_n)|)}{-\phi_{H_n}(r, p)} = 1,
\end{align}
for all  $0 < p < r < 1$. 
\end{defn}

The mean-field condition has been established for several natural hypergraph sequences. For example, it follows from results in the landmark paper of Chatterjee and Varadhan \cite{CV11}, that for any fixed graph $F$, the $F$-counting hypergraph $H_n(F)$ defined above is mean-field. In this case, the associated variational problems $\{\phi_{H_n(F)}(r, p)\}_{n \geq 1}$ themselves have a limit, which can be expressed as an optimization problem in the space of graphons (the continuum limit of graphs \cite{Lov09}). Lubetzky and Zhao \cite{LZ-dense} analyzed these variational problems, and identified the precise region of {\it replica symmetry} (set of all points $(p, r)$ where the optimization problem is uniquely minimized by the constant function $r$) for regular graphs $F$. The validity of \eqref{eq:phi} for the number of arithmetic progressions in a random set, and properties of the associated variational problem was established in \cite{BGSZ}.  Recently, Dembo and Lubetzky \cite{dembo2018large} derived the large deviations for subgraph counts in the uniform random graph model $G(n, m)$ (the uniform distribution over graphs with $n$ vertices and $m$ edges). Here, the variational problem in the rate function coincide with those studied in statistical physics in the context of constrained random graphs, where symmetry breaking configurations are often attained by block (`multi-podal') graphons (see \cite{kenyon2017multipodal,kenyon2017phase,radin2018symmetry} and the references therein).


For a general hypergraph sequence the results in \cite{CD16,eldan} can be applied to obtain different sufficient conditions on $\{H_n\}_{n\geq 1}$ for which the approximation \ref{eq:phi} holds. For instance, Bri\"et and Gopi \cite{Briet} computed the Gaussian-width a general multilinear polynomial, which combined with \cite[Theorem 5]{eldan} gives the following sufficient condition for a sequence of $s$-uniform hypergraphs $\{H_n\}_{n \geq 1}$ to be mean-field:\footnote{The first condition in \eqref{eq:suniform} controls the Gaussian width of the gradient of the function $\bm x \rightarrow t(H, \bm x)$ \cite[Corollary 6.1]{Briet}, while the second condition controls the Lipschitz constant. To verify the mean-field condition using \cite[Theorem 5]{eldan}, one also needs to check a few technical conditions related to the continuity of the variational problem, all of which can be easily verified when \eqref{eq:suniform} holds.} 
\begin{align}\label{eq:suniform}
\Delta_2(H_n) \ll  \frac{|E(H_n)|}{|V(H_n)|^{2-\frac{1}{2\ceil{\frac{s-1}{2}}}} \sqrt{\log |V(H)|}} \quad \text{and} \quad \Delta_1(H_n)=O\left(\frac{|E(H_n)|}{|V(H_n)|}\right), 
\end{align}
where $\Delta_1(H_n)$ and $\Delta_2(H_n)$ are the maximum degree and the maximum co-degree of $H_n$, respectively.\footnote{For a $s$-uniform hypergraph $H$, the degree of a vertex $v \in V(H)$ (to be denoted by $d_H(v)$) is the number of hyperedges in $H$ containing $v$, and the maximum degree $\Delta_1(H):=\max_{v \in V(H)} d_H(v)$. Similarly, for $u, v \in H$, the co-degree of $u, v$ (denoted by $d_H(u, v)$) is the number of hyperedges in $H$ containing both $u$ and $v$, and the maximum co-degree $\Delta_2(H):=\max_{u, v \in V(H)} d_H(u, v)$.} 
These assumptions are satisfied for a variety of hypergraphs, from dense $s$-uniform hypergraphs (where $E(H_n) =\Theta(|V(H_n)|^s)$) to much sparser hypergraphs such as the $r$-AP counting hypergraph (where $E(H_n) =\Theta(|V(H_n)|^2)$. 
The condition in \eqref{eq:suniform} can be significantly improved for graphs (2-uniform hypergraphs), where the mean-field condition has been extensively studied and well-understood \cite{BM17,eldan}. In this case, a sufficient condition for sequence of graphs $\{G_n\}_{n \geq 1}$ to be mean-field is 
\begin{align}\label{eq:graph}
|E(G_n)| \gg |V(G_n)| \quad \text{and} \quad  \Delta_1(G_n)=O\left(\frac{|E(G_n)|}{|V(G_n)|}\right),
\end{align}
that is, the graph $G_n$ is not `too sparse' (number of edges is much larger than the number of vertices) and not `too irregular' (maximum degree is of the same order as the average degree). In the appendix (Proposition \ref{ppn:ut_hypergraph}), we show that 
\begin{align}\label{eq:edge_hypergraph}
|E(H_n)| \gg |V(H_n)|^{s-1} \quad \text{and} \quad  \Delta_1(H_n)=O\left(\frac{|E(H_n)|}{|V(H_n)|}\right), 
\end{align}
is another simple sufficient condition for a sequence of $s$-uniform hypergraphs to be mean-field for the upper tail problem. The conditions in \eqref{eq:suniform} and \eqref{eq:edge_hypergraph} are, in general, incomparable. However, there are cases where \eqref{eq:edge_hypergraph} improves upon \eqref{eq:suniform} (see Example \ref{example:edge_condition} in Appendix \ref{sec:ut_hypergraph}). Moreover, \eqref{eq:edge_hypergraph} recovers the condition for graphs (\ref{eq:graph}) as a special case.

\begin{remark} 
In the approximation \eqref{eq:phi} one can often allow the sparsity parameter $p=p(n)$, to go to zero with $n$, at appropriate rates. Determining the optimal dependence on the sparsity parameter for a specific sequence of hypergraphs is, in general, a challenging problem. For example, for upper tails of subgraphs in the Erd\H os-R\'enyi random graph $G(n, p)$, Chatterjee and Dembo \cite[Theorem 1.2]{CD16} established the validity of \eqref{eq:phi}, using their notion of gradient complexity, for certain regimes of the sparsity parameter $p$. The dependence on $p$ was later improved by Eldan \cite{eldan}, and, very recently, Cook and Dembo \cite{CD18} and Augeri \cite{augeri1}, independently and simultaneously, established \eqref{eq:phi} for cycle counts in $G(n, p)$, under almost optimal sparsity conditions. The associated variational problems in the sparse regime was precisely analyzed in \cite{BGLZ,LZ-sparse}.
\end{remark}

In this paper, we study asymptotic properties of the variational problem \eqref{varprob} for a sequence $s$-uniform hypergraphs $\{H_n\}_{n \geq 1}$ in the fixed $p \in (0, 1)$ regime, assuming the mean-field approximation \eqref{eq:phi} holds. The following is the summary of our results:
\begin{itemize}

\item[(1)] In the fixed $p$ regime, it is notoriously difficult to solve the variational problem \eqref{varprob} explicitly for a general sequence of hypergraphs. Instead, one searches for the region of {\it replica symmetry}, that is, the set of $(p, r)$ for which the variational problem $\phi_{H_n}(p, r)$ is uniquely minimized by the constant vector $(r, r, \ldots, r) \in (0,1)^{|V(H_n)|}$ (recall $0 < p  < r < 1$). We show in Theorem \ref{repsym} that any sequence of {\it regular} $s$-uniform hypergraphs has a (universal) region of replica symmetry, which depends only on $s$, 
but not on the specific choice of the hypergraphs. Moreover, in  a sense to be made precise below, the replica symmetry region we identify for regular $s$-uniform hypergraphs is tight.

\item[(2)] We also analyze the variational problem arising in the motif frequency estimation problem, for a converging sequence of dense graphs (Section \ref{sec:dense}). In this case, the variational problems themselves have a limit which can be expressed, using the limiting graphon, as an optimization problem over the space of functions from $[0, 1] \rightarrow [0, 1]$ (Theorem \ref{nfree}), giving the {\it exact Bahadur slope} \cite{bahadur} of the corresponding estimate. 

\end{itemize}

\subsection{ Replica Symmetry for Regular Hypergraphs}

We begin with the following theorem, which identifies the precise region of universal replica symmetry for regular $s$-uniform hypergraphs. To this end, recall the definition of the mean-field variational problem $\phi_{H}(r, p)$ from \eqref{varprob}. 

\begin{thm}\label{repsym} Fix $0 < p < r < 1$. Then the following hold: 

\begin{enumerate}[(a)]

\item (Replica Symmetry) Suppose $\{H_n\}_{n \geq 1}$ is a sequence of regular $s$-uniform  hypergraphs. If the point $\left(r^s , I_p(r)\right)$ lies on the convex minorant of the function $x \mapsto I_p(x^{\frac{1}{s}})$, then $\phi_{H_n}(r, p) = I_p(r)$. Moreover, if $\{H_n\}_{n \geq 1}$ is a sequence of mean-field, regular $s$-uniform  hypergraphs, then 
$$\lim_{n\rightarrow \infty} \frac{1}{|V(H_n)|}\log \mathbb{P}\Big(T_p(H_n)\geq r^s ~\mathbb{E}\left( T_p(H_n) \right)\Big)  = -I_p(r).$$

\item (Replica Symmetry Breaking) If the point $\left(r^s , I_p(r)\right)$ does not lie on the convex minorant of  $x \mapsto I_p(x^{\frac{1}{s}})$, then there exists a sequence of mean-field, regular  $s$-uniform hypergraphs $\{H_n\}_{n \geq 1}$ (depending on $p$ and $r$) such that $\phi_{H_n}(r, p) < I_p(r)$, and 
$$\liminf_{n\rightarrow \infty} \frac{1}{|V(H_n)|}\log \mathbb{P}\Big(T_p(H_n)\geq r^s ~\mathbb{E}\left( T_p(H_n) \right)\Big)  > -I_p(r).$$
\end{enumerate}
\end{thm}

The proof of this theorem is given in Section \ref{sec:repsym}.  It shows that the variational problem \eqref{varprob} has a region of replica symmetry for any regular $s$-uniform hypergraph, which is determined by the convex minorant of the function $x \rightarrow I_p(x^\frac{1}{s})$, and this region is tight over the class of all regular $s$-uniform hypergraphs: For a regular $s$-uniform hypergraph $H$ denote by 
\begin{align}\label{eq:RH}
\cR(H):=\{(p, r): 0 < p< r< 1 \text{ and } \phi_{H}(r, p)=I_p(r)\},
\end{align} 
the set of all points where replica symmetry is preserved. Then the theorem above can be re-stated as: 
$$\bigcap_{H \in \cH_s} \cR(H)= \cC_s,$$
where $\cH_s$ is the collection of all regular $s$-uniform hypergraphs and $\cC_s$ is the set of all $(p,r)$ such that $0 < p< r< 1$ and $(r^s , I_p(r))$  lies on the convex minorant of the function $x \rightarrow I_p(x^{\frac{1}{s}})$.\footnote{Note that Theorem \ref{repsym} does not construct a regular $s$-uniform hypergraph $H$ for which $\cR(H)=\cC_s$, since the symmetry breaking construction depends on $p$ and $r$. We are only able to obtain such a construction when $r^s$ is restricted to be on the concave part of the function $x \rightarrow I_p(x^{\frac{1}{s}})$ (see Example \ref{example_concave} for details).} On the other hand, in Example \ref{bipartite}, we construct a sequence of irregular graphs which exhibit symmetry breaking everywhere, showing that the regularity assumption on the hypergraphs in Theorem \ref{repsym} is necessary for obtaining a  universal region of symmetry.

\begin{remark} For a specific sequence of hypergraphs the region of replica symmetry might be larger. For instance, if $H_n$ is the sequence of complete $s$-uniform hypergraph on $n$ vertices, it is easy to check that replica symmetry is preserved for all $0 < p < r < 1$. Another example is the replica symmetry region in the upper tails of subgraphs in the Erd\H os-Renyi random graph $G(n, p)$. To this end, given a fixed connected graph $F=(V(F), E(F))$, recall the definition of the $F$-counting hypergraph $H_n(F)$ from above: the vertex set of $H_n(F)$ is the edge set of the complete graph $K_n$, and the edge set of $H_n(F)$ is the collection of the edge sets of all copies of $F$ in $K_n$. Then $T_p(H_n(F))=N(F,  G(n, p))$, the number of copies of $F$ in $G(n, p)$. Lubetzky and Zhao \cite{LZ-dense} studied the replica symmetry region in the upper tail variational problem for $N(F,  G(n, p))$. It follows from their results that $\cR(H_n(F))\supseteq \cC_{\Delta(F)}$,  where $\Delta(F)$ is the maximum degree of $F$.\footnote{This is the exact replica symmetric region when $F$ is $D$-regular, that is, $\cR(H_n(F))= \cC_{D}$ (\cite[Theorem 1.1]{LZ-dense}).} On the other hand, Theorem \ref{repsym} above shows that  $\cR(H_n(F))\supseteq \cC_{|E(F)|}$, since $H_n(F)$ is a regular $|E(F)|$-uniform hypergraph. (Note that $\cC_{\Delta(F)} \supset \cC_{|E(F)|}$, unless $\Delta(F)=|E(F)|$, in which case $F$ is a star-graph and the two regions are the same.) 
\end{remark}

\subsection{Subgraphs in Dense Graphs} 
\label{sec:dense}

In this section, we explore the application of the general framework introduced above, in counting subgraphs of vertex-percolated graphs. Given $0 < p < 1$ and a graph $G=(V(G), E(G))$ the {\it vertex-percolated graph} $G[p]$  is the random induced subgraph $G[S]$,\footnote{Given a graph $G=(V(G), E(G))$ and $S \subseteq V(G)$, $G[S]$ denotes the induced subgraph of $G$ on the set $S$.} where $S$ is obtained by including every element of $V(G)$ independently with probability $p$. For a fixed graph $H=(V(H), E(H))$, denote by $T_p(H, G)$ the number of copies of the graph $H=(V(H), E(H))$ in $G[p]$. More formally, 
\begin{align}\label{eq:TpHG}
T_p(H, G) :=\frac{1}{|Aut(H)|}\sum_{\bm u \in V(G)_{|V(H)|}} X_{\bm u}\prod_{(a,b) \in E(H)}a_{u_a u_b}(G),
\end{align}
where:  
\begin{itemize}
          \item[--] $X_1,X_2,\cdots,X_{|V(G)|}$ are i.i.d. $\mathrm{Ber}(p)$ and $X_{\bm u} := \prod_{j=1}^{|V(H)|} X_{u_j}$, 
	 \item[--] $A(G)=((a_{uv}(G)))$ is the adjacency matrix of $G$,
	\item[--] $ V(G)_{|V(H)|}$ is the set of all $|V(H)|$-tuples ${\bm u}=(u_1,\cdots, u_{|V(H)|})\in  V(G)^{|V(H)|}$ with distinct indices.\footnote{For a set $S$, the set $S^N$ denotes the $N$-fold cartesian product $S\times S \times \cdots \times S$.} Thus, the cardinality of $V(G)_{|V(H)|}$ is $\frac{|V(G)|!}{(|V(G)|-|V(H)|)!}$, 
	\item[--] $Aut(H)$ is the {\it automorphism group} of $H$, that is, the group of permutations $\sigma$ of the vertex set $V(H)$ such that $(x, y) \in E(H)$ if and only if $(\sigma(x), \sigma(y)) \in E(H)$.
\end{itemize}  
Note that 
\begin{align}\label{eq:ETpHG}
\E(T_p(H, G)) =\frac{p^{|V(H)|}}{|Aut(H)|}\sum_{\bm u \in V(G)_{|V(H)|}} \prod_{(a,b) \in E(H)}a_{u_a u_b}(G) = p^{|V(H)|} N(H,G),
\end{align}
where $N(H, G)$ is the number of copies of $H$ in $G$.

\begin{remark}\label{rmk:estimation}  As mentioned before, the statistic \eqref{eq:TpHG} arises in the problem of estimating motif frequencies in large graphs \cite{KW}. Here, given a large graph $G_n$, the goal is to efficiently (without storing or searching over the entire graph) estimate motif frequencies (subgraph counts) of $G_n$, by making local queries on $G_n$. Klusowski and Wu \cite{KW} proposed the subgraph sampling model, where one has access to the random induced subgraph $G_n[S]$, where $S \subset V(G_n)$ is obtained by sampling each vertex in $V(G_n)$ independently with probability $p$. In this model, by \eqref{eq:ETpHG} above, $\frac{1}{p^{|V(H)|}}T_p(H, G_n)$ is an unbiased estimate of the subgraph count $N(H, G_n)$. 
 \end{remark}

The large deviation tail probabilities for $T_p(H, G_n)$ can be derived using the general theory described above. To see this, note that $T_p(H, G_n)$ can be rewritten as \eqref{eq:TpH} by defining the $|V(H)|$-uniform hypergraph which has vertex set $V(G_n)$ and a hyperedge for every copy of $H$ in $G_n$.\footnote{Technically this is a `multi-hypergraph', because a subset of $V(H)$ vertices might contain several copies of $H$ in $G_n$. However, the general theory can be easily modified to include hypergraphs with multiple edges.}  In this section, we show that the large deviation variational problem for $T_p(H, G_n)$ itself has a limit, when $G_n$ is a converging sequence of dense graphs. We begin with some preliminaries on graph limit theory. The formal statement of the result is given in Section \ref{sec:denseI}.

\subsubsection{Graph Limit Theory Preliminaries}
\label{sec:denseI} 
The theory of graph limits developed by Lov\' asz and coauthors \cite{Lov09} has received phenomenal attention over the last few years. It  connects various topics such as graph homomorphisms, Szemeredi's regularity lemma, and quasirandom graphs, and has found many interesting applications 
in statistical physics, extremal graph theory, statistics  and related areas (see \cite{BBS11,bhamidi2018weighted,CD13,radin2018symmetry,radin2013phase,yin2016,yin2016asymptotic}  and  the references therein). Here we recall the basic definitions about the convergence of graph sequences. If $F$ and $G$ are two graphs, denote the homomorphism density of $F$ into $G$ by $t(F,G) :=\frac{|\hom(F,G)|}{|V (G)|^{|V (F)|}}$,   
where  $|\hom(F,G)|$ denotes the number of homomorphisms of $F$ into $G$.

%


To define the continuous analogue of graphs, consider $\cW$ to be the space of all measurable functions from $[0, 1]^2$ into $[0, 1]$ that satisfy $W(x, y) = W(y,x)$, for all $x, y$. For a simple graph $F$ with $V (F)= \{1, 2, \ldots, |V(F)|\}$, let
\begin{align}\label{eq:tHW}
t(F,W) =\int_{[0,1]^{|V(F)|}}\prod_{(i,j)\in E(F)}W(x_i,x_j) \mathrm dx_1\mathrm dx_2\cdots \mathrm dx_{|V(F)|}.
\end{align}
\begin{defn}\cite{BCLSV08,BCLSV12,Lov09}\label{defn:graph_limit} A sequence of graphs $\{G_n\}_{n\geq 1}$ is said to {\it converge to $W \in \cW$}, if for every finite simple graph $F$, 
\begin{equation}
\lim_{n\rightarrow \infty}t(F, G_n) = t(F, W).
\label{eq:graph_limit}
\end{equation}
\end{defn}

The limit objects, that is, the elements of $\cW$, are called {\it graph limits} or {\it graphons}. A finite simple graph $G=(V(G), E(G))$ can also be represented as a graphon in a natural way: Define $W_G(x, y) =\boldsymbol 1\{(\ceil{|V(G)|x}, \ceil{|V(G)|y})\in E(G)\}$, that is, partition $[0, 1]^2$ into $|V(G)|^2$ squares of side length $1/|V(G)|$, and let $W_G(x, y)=1$ in the $(i, j)$-th square if $(i, j)\in E(G)$, and 0 otherwise. Observe that $t(F, W_G) = t(F,G)$ for every simple graph $F$ and therefore the constant sequence $G$ converges to the graph limit $W_G$. It turns out that the notion of convergence in terms of subgraph densities outlined above can be suitably metrized using the so-called {\it cut distance} (see Section \ref{nfreeproof} for details).

\subsubsection{Large Deviations of Subgraph Counts in Vertex-Percolated Dense Graphs}
\label{sec:denseII}

Consider a sequence of graphs $\{G_n\}_{n \geq 1}$ converging to a graphon $W \in \cW$ and a fixed graph $H$. In this case, the limit of the upper tail large deviation probability for $T_p(H, G_n)$ can be described as a variational problem in the space of graphons. To this end,  given a measurable function $h: [0,1] \mapsto [0,1]$, a fixed graph $H$, and graphon $W \in \cW$, define 
\begin{align}\label{eq:tHWh}
t(H,W,h) := \int\limits_{[0,1]^{|V(H)|}} \prod_{(i,j) \in E(H)} W(x_i,x_j) \prod_{i=1}^{|V(H)|} h(x_i)~\mathrm  d x_1 \cdots \mathrm  d x_{|V(H)|}.
\end{align}
(Note that when $h:=1$ is the constant function 1, $t(H,W,h) = t(H, W)$, as defined in \eqref{eq:tHW}.)

\begin{thm}\label{nfree}
Suppose $\{G_n\}_{n \geq 1}$ is a sequence of graphs converging to a graphon $W$ and $H=(V(H), E(H))$ is a fixed graph satisfying $t(H,W) > 0$. Then, for $0 < p  < r  < 1$, 
\begin{equation}\label{infvar}
\lim \limits_{n\rightarrow \infty} \frac{1}{|V(G_n)|} \log\P(T_p(H, G_n) \geq r^{|V(H)|} N(H, G_n)) = - \psi_2(H, W, r, p),
\end{equation}
with
\begin{align}
\label{eq:W}
\psi_2(H, W, r, p):=\inf\limits_{h : [0,1] \mapsto [0,1]} \left\{\int_0^1 I_p(h(x)) \mathrm d x : t(H,W,h) \geq r^{|V(H)|} t(H,W)\right\},
\end{align}
where the infimum is taken over the set of all measurable functions $h: [0,1]\mapsto [0,1]$. 
\end{thm}

The proof of this result and various examples are given in Section \ref{nfreeproof}. Note that the assumption $t(H,W) > 0$ implies that the limiting density of $H$ in $G_n$ is non-vanishing, and ensures, among other things, that the mean-field assumptions are satisfied. This also gives the exact Bahadur slope \cite{bahadur} of the unbiased estimate $\frac{1}{p^{|V(H)|}}T_p(H, G_n)$ (recall Remark \ref{rmk:estimation}) for the subgraph count $N(H, G_n)$, in the subgraph-sampling model (see Remark \ref{rmk:2sided} for details).

\section{Proof of Theorem \ref{repsym}}
\label{sec:repsym}

\noindent\textbf{\textit{Proof of Replica-Symmetry}}: The proof for the replica symmetry case relies on the following lemma. 

\begin{lem}\label{lm:ds} For any $s$-uniform $d$-regular hypergraph $H=(V(H), E(H))$, and $\bm x=(x_1, x_2,$ $\ldots, x_{|V(H)|}) \in [0, 1]^{|V(H)|}$,  $\frac{1}{|E(H)|}t(H, \bm x) \leq \frac{1}{|V(H)|} \sum_{v=1}^{|V(H)|} x_v^s$, where $t(H, \bm x)$ is as in Definition \ref{defn:approx}. 
\end{lem}

\begin{proof} Note that it suffices to show $t(H,\bm x) \leq \frac{d}{s} \sum_{v=1}^{|V(H)|} x_v^s$, for $\bm x \in [0,1]^{|V(H)|}$, since $|E(H)|=\frac{d}{s}|V(H)|$. To  this end, define the function, $$\eta(\bm x) = \frac{d}{s} \sum_{v=1}^{|V(H)|} x_v^s -  t(H, \bm x),$$
This is a smooth function on the compact set $[0, 1]^{|V(H)|}$, and hence the minimum of $\eta(\cdot)$ is attained at some point $\bm z \in [0,1]^{|V(H)|}$. Therefore, to prove the lemma, it suffices to show that $\eta(\bm z) \geq 0$. 

To show this, let $S:= \{v \in V(H): z_v \in (0,1)\}$. Then, $\frac{\partial}{\partial z_v}\eta(\bm z) = 0$, for all $v \in S$. This implies, 
\begin{equation}\label{diffres}
d {z_v}^{s-1} = \sum_{e\in E(H): v \in e} ~\prod_{u \in e\setminus\{v\}} z_u, \quad \textrm{for all}~v \in S.
\end{equation}
Hence, using \eqref{diffres}, 
$$\eta(\bm z) = \frac{1}{s} \sum_{v=1}^{|V(H)|} \left(d {z_v}^s - \sum_{e\in E(H): v \in e} ~\prod_{u\in e} z_u\right) = \frac{1}{s}\sum_{v \in V(H)\setminus S:z_v= 1}\left(d - \sum_{e\in E(H): v \in e} ~\prod_{u\in e} z_u\right)\geq 0,$$ 
where the last step uses $\sum_{e\in E(H): v \in e} ~\prod_{u\in e} z_u \leq |\{e\in E(H): v \in e\}| =d$ , for all $v \in V(H)$. 
\end{proof}

To show replica symmetry, suppose $(r^s, I_p(r))$ lies on the convex minorant the function $J_p(x):= I_p(x^{\frac{1}{s}})$. Let $\hat{J}_p$ denote the convex minorant of the function $J_p$. Note that $\hat{J}_p$ is increasing on the interval $[p^s,1]$, and $J_p \geq \hat{J}_p$ on $[0,1]$. Hence, for $\bm x \in [0, 1]^{|V(H)|}$ such that $\frac{1}{|E(H)|}t(H, \bm x) \geq r^s$, 
\begin{align}
\frac{1}{|V(H)|}\sum_{v=1}^{|V(H)|} I_p(x_v) = \frac{1}{|V(H)|}\sum_{v=1}^{|V(H)|} J_p\left(x_v^s\right) & \geq \frac{1}{|V(H)|}\sum_{v=1}^{|V(H)|} \hat{J}_p\left(x_v^s\right) \tag*{(using $J_p \geq \hat{J}_p$)} \nonumber \\ 
& \geq  \hat{J}_p\left(\frac{1}{|V(H)|}\sum_{v=1}^{|V(H)|} x_v^s\right) \tag*{(Jensen's inequality)}  \nonumber \\ 
& \geq \hat{J}_p\left(r^s\right) \tag*{(using Lemma \ref{lm:ds})}  \nonumber \\ 
\label{str1} & = I_p(r), 
\end{align} 
The proof of replica symmetry now follows, since there is equality everywhere in \eqref{str1}, if  $x_v = r$, for all $v \in V(H)$. \\

\noindent\textbf{\textit{Proof of Replica Symmetry Breaking}}: 
Here, suppose $(r^s, I_p(r))$ does not lie on the convex minorant of  the function $x \rightarrow I_p(x^{\frac{1}{s}})$. Then there exist $0 \leq r_1 < r < r_2 \leq 1$ such that the point $(r^s, I_p(r))$ lies strictly above the line segment joining $(r_1^s, I_p(r_1))$ and $(r_
2^s, I_p(r_1))$, that is, there exists $0 < \lambda < 1$ such that 
$$\lambda r_1^s + (1-\lambda) r_2^s = r^s \quad \text{and} \quad \lambda I_p(r_1) + (1-\lambda) I_p(r_2) < I_p(r).$$
By continuity, $\lambda' I_p(r_1) + (1-\lambda') I_p(r_2) < I_p(r)$ for all $\lambda'$ in a neighborhood of $\lambda$, which means, there exists $0 < \lambda_0 < 1$ such that 
$$\lambda_0 r_1^s + (1-\lambda_0) r_2^s > r^s \quad \text{and} \quad \lambda_0 I_p(r_1) + (1-\lambda_0) I_p(r_2) < I_p(r).$$
Therefore, choose $M$ large enough so that $M\lambda_0 \geq 1$, and
$$\lambda_0 r_1^s + (1-\lambda_0) r_2^s - \frac{1}{M} > r^s \quad and \quad \lambda_0 I_p(r_1) + (1-\lambda_0) I_p(r_2) + \frac{C_p}{M} < I_p(r),$$
where $C_p:=\log(1/p)+\log(1/q)$ (with $q:=1-p$). (Note, this choice of $M$ depends on $p$ and $r$.)

Now, let $H_n$ be the disjoint union of $M$ complete $s$-uniform hypergraphs on $\ceil{\frac{n}{M}}$ vertices each. This is a dense hypergraph, that is $|E(H_n)|=\Theta(n^s)$, and, hence, satisfies  assumptions \eqref{eq:suniform}, which implies that $\{H_n\}_{n \geq 1}$ is a sequence of mean-field, regular $s$-uniform hypergraphs. Assuming that the vertices of $H_n$ are labelled $\{1, 2, \ldots, M  \ceil{\frac{n}{M}} \}$, define $\bm x = (x_1, x_2, \ldots, x_{|V(H_n)|}) \in [0, 1]^{|V(H_n)|}$ as 
$$x_{j}=
\left\{
\begin{array}{ccc}
r_1 & \text{ for } & j \in \{1, 2, \ldots, K \ceil{\frac{n}{M}} \}, \\ 
r_2 & \text{ for } & j \in \{ K \ceil{\frac{n}{M}} +1, K \ceil{\frac{n}{M}} +2, \ldots,  M  \ceil{\frac{n}{M}}\},
\end{array}
\right.
$$
where $ K=\lfloor M \lambda_0 \rfloor$. Then using $\lambda_0 -\frac{1}{M} \leq \frac{K}{M} \leq \lambda_0$, $r_1 \leq 1$, and $t(H, \bm x)$ as in Definition \ref{defn:approx}, 
\begin{align*}
\frac{t(H_n,\bm x)}{|E(H_n)|}  =\frac{1}{M {\ceil{\frac{n}{M}} \choose s} } \left(K {\ceil{\frac{n}{M}} \choose s}r_1^s+  {\ceil{\frac{n}{M}} \choose s}  (M-K)  r_2^s \right) 
& \geq \lambda_0 r_1^s+ (1- \lambda_0) r_2^s -\frac{r_1^s}{M} \nonumber \\ 
& > r^s. 
\end{align*}
On the other hand, the entropy satisfies, 
\begin{align*}
\frac{1}{|V(H_n)|} \sum_{u=1}^{|V(H_n)|} I_p(x_u)  & = \frac{1}{M  \ceil{\frac{n}{M}}} \left\{ K \left \lceil \frac{n}{M} \right \rceil I_p(r_1) +  \left( M  \left \lceil \frac{n}{M} \right \rceil -  K \left \lceil \frac{n}{M} \right \rceil \right) I_p(r_2) \right\} \nonumber \\ 
& =  \frac{K}{M} I_p(r_1) +  \left( 1  -  \frac{K}{M}  \right) I_p(r_2)   \nonumber \\ 
& \leq \lambda_0 I_p(r_1) + (1-\lambda_0) I_p(r_2) + \frac{I_p(r_2)}{M} \tag*{(using $\lambda_0 -\frac{1}{M} \leq \frac{K}{M} \leq \lambda_0$)} \nonumber \\ 
& \leq \lambda_0 I_p(r_1) + (1-\lambda_0) I_p(r_2)  + \frac{C_p}{M} \tag*{(recall $C_p=\log(1/p)+\log(1/q)$)} \nonumber \\ 
& < I_p(r), 
\end{align*}
which completes the proof of Theorem \ref{repsym}. \hfill $\Box$ \\  

Note that in the construction above the hypergraphs $\{H_n\}_{n \geq 1}$ which break symmetry at $(p, r)$, depend on $p$ and $r$. Whether it is possible to obtain a sequence of regular, mean-field $s$-uniform hypergraphs not depending on $p$ and $r$, which breaks symmetry whenever $(r^s, I_p(r))$ does not lie on the convex minorant of $x \rightarrow I_p(x^\frac{1}{s})$, remains open. In the example below, we construct a sequence of regular, mean-field $s$-uniform hypergraphs (not depending on $p$ and $r$) which breaks symmetry whenever the point $(r^s, I_p(r))$ lies on the concave part of $x\rightarrow I_p(x^\frac{1}{s})$.\footnote{If $s\geq 1$, the function $J_p(x) = I_p(x^{\frac{1}{s}})$ is convex if and only if $p \geq p_0(s) := \frac{s-1}{s-1 + e^{s/(s-1)}}$. On the other hand, if $0 < p < p_0(s)$, then the function $J_p(x)$ has exactly two inflection points (both to the right of $x = p^s$), with a region of concavity in the middle (see \cite[Lemma A.1]{LZ-dense} for details).}

\begin{example}\label{example_concave} Assume that the point $\left(r^s, I_p(r)\right)$ lies on the strictly concave part of the curve of $J_p(x)=I_p(x^{\frac{1}{s}})$. Hence, there exist two points $a, b$ in a neighborhood of $r^s$ such that 
\begin{equation}\label{strictconcave}
\frac{a+b}{2} = r^s \quad \text{and} \quad J_p(r^s) > \frac{1}{2}J_p(a) + \frac{1}{2}J_p(b).
\end{equation}
Let $n=2N$ be even, and let $H_n$ be the disjoint union of two complete $s$-uniform hypergraphs with $N$ vertices each, and vertices labelled labelled $\{1,2,\cdots,N\}$ and $\{N+1,N+2,\cdots,2N\}$, respectively.  Define $\bm x = (x_1, x_2, \ldots, x_{n}) \in [0, 1]^{n}$ as: 
$x_1 = \cdots = x_{N} := a^{\frac{1}{s}}$  and $x_{N+1} = \cdots x_{n} := b^{\frac{1}{s}}$. It is easy to check that $\bm x := (x_1, \cdots, x_n)$ belongs to the constraint space of \eqref{varprob}, and by \eqref{strictconcave}, $\frac{1}{n}\sum_{i=1}^n I_p(x_i) = \frac{1}{2}J_p(a) + \frac{1}{2}J_p(b) < I_p(r)$, which shows replica-symmetry-breaking.
\end{example}

\section{Proof of Theorem \ref{nfree}}\label{nfreeproof}

We begin with a few definitions from graph limit theory. The notion of graph convergence in terms of subgraph densities in Definition \ref{defn:graph_limit} can be metrized using the {\it cut-metric}, which we recall below: 

\begin{defn}\cite{Lov09}\label{defn:Wconvergence}
The {\it cut-distance} between 2 graphons  $W_1, W_2 \in \cW$, is defined as 
\begin{align}\label{eq:Wconvergence}
||W_1-W_2||_\square:=\sup_{f, g: [0, 1] \rightarrow [-1, 1]}\left|\int_{[0, 1]^2} \left(W_1(x, y)-W_2(x, y)\right) f(x) g(y) \mathrm dx \mathrm dy \right|. 
\end{align} 
The {\it cut-metric} between 2 graphons  $W_1, W_2 \in \cW$, is defined as 
$$\delta_{\square}(W_1, W_2):= \inf_{\psi}||W_1^{\psi}-W_2||_\square,$$
with the infimum taken over all measure-preserving bijections $\psi: [0, 1] \rightarrow [0, 1]$, and  $W_1^\psi(x, y):= W_1(\psi(x), \psi(y))$, for $x, y \in [0, 1]$. 
\end{defn}

Hereafter, we assume that $\{G_n\}_{n \geq 1}$ is a sequence of graphs converging to a graphon $W$, which is equivalent to $\delta_\square(W_{G_n}, W) \rightarrow 0$ \cite[Theorem 3.8]{BCLSV08}. Now, for a fixed graph $H$, define the multi-hypergraph $M_{G_n}(H)$ with vertex set $V(G_n)$ and a hyperedge for every copy of $H$ in $G_n$. The assumption $t(H, W)>0$ implies that $|E(M_{G_n}(H))|=\Theta(N(H, G_n))=\Theta(|V(G_n)|^{|V(H)|})$, and 
\begin{align*}
\Delta_2(M_{G_n}(H))=\Theta\left(|V(G_n)|^{|V(H)|-2}\right) \quad \text{and} \quad \Delta_1(M_{G_n}(H))=\Theta\left(|V(G_n)|^{|V(H)|-1}\right), 
\end{align*} 
that is, the sequence $\{M_{G_n}(H)\}_{n \geq 1}$ satisfies assumption \eqref{eq:suniform}. Therefore, $\{M_{G_n}(H)\}_{n \geq 1}$ is mean-field, and by Definition \ref{defn:approx}, for every $0 < p < r < 1$:
$$\lim_{n\rightarrow \infty} \frac{\frac{1}{|V(G_n)|}\log \mathbb{P}\left(T_p(H,G_n)\geq r^{|V(H)|} N(H, G_n)\right) }{-\psi_{G_n}(r, p, H)} = 1.$$ 
where 
\begin{equation}\label{graphvarprob}
\psi_{G_n}(r, p, H) = \frac{1}{|V(G_n)|}\inf_{\bm x \in [0,1]^{|V(G_n)|}} \left\{\sum_{v=1}^{|V(G_n)|} I_p(x_v) :~N(H,G_n,\bm x) \geq r^{|V(H)|} N(H,G_n)\right\}.
\end{equation}
with 
\begin{itemize}
\item[--] $\bm x=(x_1, x_2, \ldots, x_{|V(G_n)|}) \in [0,1]^{|V(G_n)|}$ and $\bm x_{\bm u} := \prod_{j=1}^{|V(H)|} x_{u_j}$, and 

\item[--] $N(H,G_n,\bm x) = \frac{1}{|Aut(H)|}\sum_{\bm u \in V(G_n)_{|V(H)|}} \bm x_{\bm u}\prod_{(a,b) \in E(H)}a_{u_a, u_b}(G_n).$ 
\end{itemize}

The argument above shows that to prove Theorem \ref{nfree}, it suffices to prove that the limit of \eqref{graphvarprob} equals \eqref{eq:W}. To this end, suppose $\{G_n\}_{n \geq 1}$ is a sequence of graphs such that $\lim_{n \rightarrow \infty}\delta_\square(W_{G_n}, W) = 0$. Then it follows from \cite[Lemma 5.3]{BCLSV08} that there exists a permutation $\sigma: [|V(G_n)|] \rightarrow  [|V(G_n)|]$ such that $||W_{G_n^\sigma}-W||_\square \rightarrow 0$, where $G_n^\sigma$ is a graph obtained by relabelling the vertices of $G_n$ by $\sigma$. Moreover, the variational problem \eqref{graphvarprob}  is invariant under vertex permutations, that is, $\psi_{G_n}(r, p, H)=\psi_{G_n^\sigma}(r, p, H)$. This implies, to derive the limit of \eqref{graphvarprob}, we can, without loss of generality, assume $||W_{G_n}-W||_\square \rightarrow 0$. 

We begin with the following lemma, which follows by a telescoping argument identical to the proof of \cite[Theorem 3.7]{BCLSV08}. We omit the details. 
	
\begin{lem}\label{borgsthm} Let $G_n$ be a sequence of graphs such that $||W_{G_n}-W||_{\square} \rightarrow 0$. Then, for any fixed graph $H=(V(H), E(H))$, 
$$\lim_{n\rightarrow \infty} \sup\limits_{f_1,\cdots,f_{|V(H)|}} \Bigg|\int\limits_{[0,1]^{|V(H)|}} \left(\prod_{(i,j) \in E(H)} W_{G_n}(x_i,x_j) - \prod_{(i,j) \in E(H)} W(x_i,x_j)\right)\prod_{i=1}^{|V(H)|} f_i(x_i)~\mathrm  d x_i \Bigg| = 0, $$ 
where the supremum runs over all measurable functions $f_1,\cdots,f_{|V(H)|} : [0,1] \mapsto [0,1]$. 
\end{lem}

Given a function $h: [0,1] \mapsto [0,1]$ and a graphon $W \in \cW$, recall the definition of $t(H, W, h)$ from \eqref{eq:tHWh}. Also, denote by $\cT_n $ the class of all functions $h : [0,1] \mapsto [0,1]$, which are constant on the intervals $\left(\frac{a-1}{|V(G_n)|},\frac{a}{|V(G_n)|}\right]$, for every $1 \leq a \leq |V(G_n)|$. Then, recalling \eqref{graphvarprob}, 
\begin{align}\label{step1}
&\psi_{G_n} (H, r, p) = \nonumber \\ 
& \inf\limits_{h \in \cT_n} \Bigg\{\int_0^1 I_p(h(x)) \mathrm d x : t(H,W_{G_n},h) - R\left(H,G_n,h\right)\geq r^{|V(H)|} \left(t(H,W_{G_n})-R\left(H,G_n\right)\right)\Bigg\}, 
\end{align}
where 
$$R\left(H,G_n, h\right) = \sum_{\bm u \in [|V(G_n)|]^{|V(H)|}\big\setminus [|V(G_n)|]_{|V(H)|}} \prod_{(a,b)\in E(H)} a_{u_au_b}(G_n)\prod_{i=1}^{|V(H)|}\int_{\frac{u_i-1}{|V(G_n)|}}^{\frac{u_i}{|V(G_n)|}} h(y)\mathrm d y,$$
and $R(H,G_n)=R(H,G_n, 1)$, which is $R(H,G_n, h)$, when $h$ is the constant function $1$. The adjustment by $R(H,G_n, h)$ is required to move from the sum over all indices $[|V(G_n)|]^{|V(H)|}$ in $t(H, W_{G_n}, h)$ to the sum over distinct indices $[|V(G_n)|]_{|V(H)|}$ in $N(H, G_n, \bm x)$ (recall \eqref{graphvarprob}). (Note that $R\left(H,G_n, h\right) = 0$ if $H$ is a clique.) However, the asymptotic contribution of this adjustment is small: 
\begin{align}\label{remgoestozero}
\sup_{h: [0,1]\mapsto [0,1]} R(H,G_n,h) \leq \frac{1}{|V(G_n)|^{|V(H)|}} \left|[V(G_n)]^{|V(H)|}\big\setminus [V(G_n)]_{|V(H)|}\right| \leq  \frac{\binom{|V(H)|}{2}}{|V(G_n)|}. 
\end{align}

\begin{lem} If $G_n$ is a sequence of graphs  such that $||W_{G_n}-W||_\square \rightarrow 0$, then for every fixed graph $H$ and $0 < p < r < 1$, 
\begin{equation}\label{step2}
\lim_{n\rightarrow \infty} \left\{\psi_{G_n} (H, r, p)  - A_n(H, r, p)\right\} = 0, 
\end{equation} 
where $A_n(H, r, p) :=  \inf_{h \in \cT_n} \left\{\int_0^1 I_p(h(x)) \mathrm d x : t(H,W,h) \geq r^{|V(H)|} t(H,W)\right\}$.
\end{lem}

\begin{proof} 
Fix $r \in (p,1)$. For each $n \geq 1$, choose $h_n \in \cT_n$ such that 
\begin{align}\label{eq:approx_I}
t(H,W,h_n) \geq r^{|V(H)|} t(H,W) \quad \textrm{and}\quad \int_{0}^1 I_p(h_n(x)) \mathrm d x < A_n(H,r,p) + \frac{1}{n}.
\end{align} 
By Lemma \ref{borgsthm}, $t(H,W_{G_n},h_n) - t(H,W,h_n) \rightarrow 0$ and $t(H,W_{G_n}) \rightarrow t(H,W)$, as $n \rightarrow \infty$. 
Then there exists $\varepsilon \in (0, r/p)$ such that, by \eqref{remgoestozero}, 
$$t(H,W_{G_n},h_n) - R(H,G_n,h_n) > t(H,W,h_n) - \varepsilon \quad \text{and} \quad \frac{ t(H,W) }{2}< t(H,W_{G_n}) < \frac{r^{|V(H)|} t(H,W)}{(r-\varepsilon p)^{|V(H)|}},$$
for all large $n$. Now, define $\varepsilon':=\frac{2\varepsilon}{p^{|V(H)|} t(H,W)}$. Then 
\begin{align*}
t(H,W_{G_n},h_n) - R(H,G_n,h_n) &> r^{|V(H)|} t(H,W) - \varepsilon \tag*{(by \eqref{eq:approx_I})} \\ 
&> (r-\varepsilon p)^{|V(H)|} t(H,W_{G_n}) - \varepsilon  \\
&>  (r-\varepsilon p)^{|V(H)|}t(H,W_{G_n}) -  \varepsilon' p^{|V(H)|} t(H,W_{G_n}) \\
&\geq \left(\left (r/p-\varepsilon \right)^{|V(H)|} - \varepsilon' \right)p^{|V(H)|} \left(t(H,W_{G_n})-R(H,G_n)\right),
\end{align*}
by choosing $\varepsilon$ small enough so that  $\left (r/p-\varepsilon \right)^{|V(H)|} - \varepsilon'>0$. This implies, recalling \eqref{step1} and \eqref{eq:approx_I}, 
\begin{equation}\label{firstless}
\psi_{G_n}\left(H,r_\varepsilon,p\right) \leq \int_0^1 I_p(h_n(x)) \mathrm d x < A_n(H,r,p) + \frac{1}{n}.
\end{equation}
where $r_\varepsilon=\left(\left(r/p-\varepsilon \right)^{|V(H)|} - \varepsilon' \right)^{\frac{1}{|V(H)|}} p$. Note that $r_\varepsilon \rightarrow r$, as $\varepsilon \rightarrow 0$, and by arguments similar to the proof of \cite[Lemma 5.8]{CD16}, it follows that 
$$\psi_{G_n}\left(H,r_\varepsilon, p\right) \geq \psi_{G_n} (H,r,p) + o(1),$$ 
where the $o(1)$-term depends on $p, r, H$, and $\varepsilon$, but not on $n$, and goes to zero as $\varepsilon \rightarrow 0$. This implies $\limsup_{n \rightarrow \infty} \left\{\psi_{G_n} (H,r,p) - A_n(H,r,p)\right\} \leq 0$. The other direction can be proved similarly. 
\end{proof}

Next, let $\mathcal{Q}$ denote the class of all measurable functions $h: [0,1] \mapsto [0,1]$, such that $h$ is continuous at every irrational point in $[0,1]$. Then define 
$$L_W(H, r, p) := \inf\limits_{h \in \mathcal{Q}}  \Bigg\{\int_0^1 I_p(h(x)) \mathrm d x : t(H,W,h) \geq r^{|V(H)|} t(H,W)\Bigg\}.$$
Note that $\cT_n \subset \mathcal{Q}$, for each $n \geq 1$, which gives $\liminf_{n\rightarrow\infty} A_n(H, r, p) \geq L_W(H, r, p).$ Now, fixing $\varepsilon > 0$, choose $h \in \mathcal Q$ such that 
\begin{align}\label{eq:I}
t(H,W,h) \geq r^{|V(H)|} t(H,W) \quad \text{and} \quad \int_0^1 I_p(h(x)) \mathrm dx < L_W(H, r, p) + \varepsilon.
\end{align}
For $n \geq 1$, define $$h_n(x) = \sum_{a=1}^{|V(G_n)|} \sup\limits_{y \in \left(\frac{a-1}{|V(G_n)|},\frac{a}{|V(G_n)|}\right]} h(y)~\bm{1}\left\{x \in \left(\frac{a-1}{|V(G_n)|},\frac{a}{|V(G_n)|} \right]\right\}.$$ 
Clearly, $h_n \geq h$, and hence, $t(H,W, h_n) \geq r^{|V(H)|} t(H,W)$, for all $n \geq1$. Moreover, since $h_n \in \cT_n$, 
\begin{align}\label{eq:II}
A_n(H, r, p) \leq \int_0^1 I_p(h_n(x)) \mathrm d x, \quad \text{ for all } n \geq 1.
\end{align}
Now,  since $h_n(x) \rightarrow h(x)$ for every irrational $x \in [0,1]$, and $I_p$ is a bounded, continuous function on $[0,1]$, $\lim_{n\rightarrow \infty}\int_0^1 I_p(h_n(x)) \mathrm d x = \int_0^1 I_p(h(x)) \mathrm dx,$ by the dominated convergence theorem. Hence, by taking limits in \eqref{eq:II} and using \eqref{eq:I}, we have $\limsup_{n\rightarrow\infty} A_n(H, r, p) < L_W(H, r, p) + \varepsilon$. Finally, since $\varepsilon>0$ is arbitrary, by \eqref{step2}, we get 
\begin{equation}\label{fin}
\lim_{n\rightarrow\infty} \psi_{G_n}(H, r, p) =  L_W(H, r, p).
\end{equation}

Therefore, to complete the proof of Theorem  \ref{nfree}, it suffices to show $L_W(H, r, p)=\psi_2(H, W, r, p)$ (recall \eqref{eq:W}). Clearly, $\psi_2(H, W, r, p) \leq L_W(H, r, p)$. For the other direction, let $\varepsilon \in (0, r/p)$ and take a measurable function $h : [0,1] \mapsto [0,1]$ such that 
\begin{align}\label{eq:III}
t(H,W,h) \geq r^{|V(H)|} t(H,W) \quad \text{and} \quad \int_0^1 I_p(h(x)) \mathrm dx  < \psi_2(H,W, r, p) + \varepsilon
\end{align}
By standard measure theoretic arguments,  there exists a continuous function $g: [0,1] \mapsto [0,1]$ such that: $\int_0^1 I_p(g(x)) \mathrm dx < \int_0^1 I_p(h(x)) \mathrm d x  + \varepsilon$ and 
$$t(H,W, g) \geq  t(H,W,h) - \varepsilon^{|V(H)|} p^{|V(H)|} t(H,W) \geq (r^{|V(H)|}-(\varepsilon p)^{|V(H)|} ) t(H,W),$$
where the last step uses \eqref{eq:III}. Hence, defining $r_\varepsilon=(r^{|V(H)|}-(\varepsilon p)^{|V(H)|})^{\frac{1}{|V(H)|}}$, gives 
\begin{align}\label{eq:IV}
L_W(H, r _\varepsilon, p) \leq \int_0^1 I_p(g(x)) \mathrm dx < \int_0^1 I_p(h(x)) \mathrm d x  + \varepsilon < \psi_2(H, W, r, p) + 2\varepsilon
\end{align} 
where the last step uses \eqref{eq:III}. Finally, since $r_\varepsilon \rightarrow r$, as $\varepsilon \rightarrow 0$, by arguments similar to the proof of \cite[Lemma 5.8]{CD16}, $$L_W(H, r_\varepsilon, p) \geq L_W(H, r, p) + o(1),$$  
where the $o(1)$-term goes to zero as $\varepsilon \rightarrow 0$. This combined with \eqref{eq:IV} shows that $L_W(H, r, p) = \psi_2(H, W, r, p)$, which completes the proof of Theorem \ref{nfree}. \hfill $\Box$ \\

\begin{remark}\label{rmk:2sided} By arguments similar to the proof of Theorem \ref{nfree}, an analogous variational problem can be established for the two-sided tail probability.  More precisely, if $\{G_n\}_{n \geq 1}$ is a sequence of graphs converging in cut-metric to a graphon $W \in \cW$, and $H=(V(H), E(H))$ is a fixed graph satisfying $t(H,W) > 0$, then for any fixed $\varepsilon >0$, it follows that 
\begin{equation*}
\lim \limits_{n\rightarrow \infty} \frac{1}{|V(G_n)|}\log \P\left(\left|\frac{1}{p^{|V(H)|}}T_p(H, G_n) - N(H, G_n)\right| \geq \varepsilon\right) = -\gamma_2(H, W, \varepsilon, p),
\end{equation*}
where $$\gamma_2(H, W, \varepsilon, p):=\inf_{h : [0,1] \mapsto [0,1]} \left\{\int_0^1 I_p(h(x)) \mathrm d x : \left|\frac{t(H,W,h)}{p^{|V(H)|}}- t(H,W)\right| \geq \varepsilon \right\},$$  is the {\it exact Bahadur slope} \cite{bahadur} of $\frac{1}{p^{|V(H)|}}T_p(H, G_n)$, the estimate of $N(H, G_n)$ in the subgraph sampling model described in Remark \ref{rmk:estimation}. 
\end{remark}

The variational problem in \eqref{eq:W} reduces to a finite dimensional optimization problem,  if the limiting graphon $W$ is a block function (constant on finitely many blocks). These functions  are dense in the space of graphons and arise naturally as limits of stochastic block models.  

\begin{remark}\label{rem:block_variational} (Block Graphons) Suppose $\{G_n\}_{n \geq 1}$ is a sequence of graphs converging in cut-metric to the following $K$-block graphon: 
$$W(x, y) := \sum_{i=1}^K \sum_{j=1}^K b_{ij} \bm{1}\{ x \in A_i  \text{ and }  y \in A_{j}\},$$ 
where $B:=((b_{ij}))_{1 \leq i, j  \leq K}$ is a non-zero $K \times K$ symmetric matrix with entries in $[0, 1]$, and $A_1, A_2, \ldots, A_K$ form a measurable partition of $[0,1]$, with $\lambda(A_i)>0$, for all $1 \leq i \leq K$ (here $\lambda(\cdot)$ denotes the Lebesgue measure on $[0, 1]$). Then, for any fixed graph $H=(V(H), E(H))$, the homomorphism density 
$$t(H,W) = \sum_{(u_1, \ldots, u_{|V(H)|}) \in [K]^{|V(H)|}}\prod_{(a,b)\in E(H)} b_{u_a u_b}\prod_{i=1}^{|V(H)|}\lambda(A_{u_i}).$$
Further, for $\bm x=(x_1, x_2, \ldots, x_K) \in [0,1]^K$, define
$$t(H,W,\bm x) := \sum_{(u_1, \ldots, u_{|V(H)|}) \in [K]^{|V(H)|}}\prod_{(a,b)\in E(H)}b_{u_au_b}\prod_{i=1}^{|V(H)|}\lambda(A_{u_i})x_{u_i}.$$
Then the RHS of \eqref{eq:W} becomes, 
\begin{equation}\label{eq:blockfinvar}
\inf_{\bm x = (x_1, x_2, \ldots, x_K) \in [0,1]^K} \left\{\sum_{i=1}^K \lambda(A_i)I_p(x_i)~:~t(H,W,\bm x) \geq r^{|V(H)|} t(H,W)\right\},
\end{equation}
which is a finite dimensional optimization problem with $K$ variables. To see \eqref{eq:blockfinvar},  note that for any $h: [0, 1] \rightarrow [0, 1]$ in the constraint space of \eqref{eq:W}, the vector $\bm x \in [0,1]^K$ defined by $$x_i := \frac{1}{\lambda(A_i)} \int_{A_i} h(z) \mathrm dz, \quad \text{ for } 1\leq i \leq K,$$
belongs to the constraint space of \eqref{eq:blockfinvar}. Then by the convexity of $I_p(\cdot)$, $$\sum_{i=1}^K \lambda(A_i)I_p(x_i) \leq \int_0^1 I_p(h(z)) \mathrm d z,$$ which shows that  \eqref{eq:blockfinvar} is at most the RHS of \eqref{eq:W}. Conversely, for some $\bm x$ in the constraint space of \eqref{eq:blockfinvar}, the function $h(z):= \sum_{i=1}^K x_i \bm{1}\{z \in A_i\}$ belongs to the constraint space of \eqref{eq:W} and 
$\int_0^1 I_p(h(z)) \mathrm dz = \sum_{i=1}^K \lambda(A_i)I_p(x_i)$. Hence, the RHS of \eqref{eq:W} is equal to \eqref{eq:blockfinvar}.  
\end{remark}

We conclude with an example of a sequence of non-regular graphs which exhibits replica symmetry breaking, for all $0 < p < r < 1$.

\begin{example}\label{bipartite} (Bipartite Graphs) Consider the sequence of complete bipartite graphs $K_{m,n}$, with $m/(m+n) \rightarrow \alpha \in (0,1)$. In this case, the limiting graphon $W$ is a two-block function with block sizes $\alpha$ and $(1-\alpha)$, with $0$ on the diagonal blocks and 1 on the off-diagonal blocks. Then by Theorem \ref{nfree}, for  $0 < p < r < 1$, 
$$\frac{1}{m+n}\log \P(T_p(K_2, K_{m, n}) \geq r^2 mn ) \rightarrow -\psi_2(K_2, W, r, p).$$
In this case, $W$ is 2-block, hence, by Remark \ref{rem:block_variational}, 
\begin{equation}\label{compbip}
\psi_2(K_2, W, r, p)=\inf_{(x, y) \in [0,1]^2} \left\{\alpha I_p(x) + (1-\alpha)I_p(y)~: xy  \geq r^2 \right\}.
\end{equation}
\begin{itemize}	

\item $\alpha =\frac{1}{2}$: Then any $(x, y)$ in the constraint space of \eqref{compbip} satisfies $\frac{x+y}{2}\geq \sqrt{xy} \geq r$, and by the convexity of $I_p$ and the fact that it is increasing on the interval $(p,1)$, it follows that $\frac{1}{2}I_p(x) + \frac{1}{2}I_p(y) \geq I_p\left(r\right)$, showing replica symmetry, for all values of $0 < p < r < 1$. 
	
\item $\alpha \neq \frac{1}{2}$: In this case,  the graph sequence $K_{m, n}$ is irregular. Note that \eqref{compbip} equals the minimum of the function:
$$g(x) := \alpha I_p(x) + (1-\alpha)I_p\left(\frac{r^2}{x}\right), \quad \text{over} \quad x \in \left[r^2,1\right].$$ 
Note that $g$ is differentiable on $\left(r^2, 1\right)$ and $g'\left(r\right) \neq 0$, showing replica-symmetry-breaking  for all values of $0 < p < r < 1$.
\end{itemize} 
\end{example}

\bibliographystyle{abbrv}
\bibliography{ldp_ref}

\appendix

\section{A Simple Mean-Field Condition}
\label{sec:ut_hypergraph}

In this section we derive a simple sufficient condition for a sequence of  $s$-uniform hypergraphs to be mean-field for the upper tail problem (recall Definition \ref{defn:approx}), using the framework of non-linear large deviations developed in \cite{eldan}. To this end, we need a few definitions: The \textit{Lipschitz constant} of a function $h: \{0,1\}^n \mapsto \mathbb{R}$, is defined as: $$\mathrm{Lip}(h) = \max_{i\in[n],\bm y\in \{0,1\}^n} |\partial_i h(\bm y)|,$$ where 
$$\partial_i h(\bm y)=h(y_1, \ldots, y_{i-1}, 1, y_{i+1}, \ldots, y_n)-h(y_1, \ldots, y_{i-1}, 0, y_{i+1}, \ldots, y_n),$$ 
denotes the discrete partial derivative of $h(\bm y)$ with respect to the $i$-th coordinate. The \textit{Gaussian-width} of a set $A \subseteq \mathbb{R}^n$ is defined as: $$\mathbf{GW}(A) = \E\left[\sup_{\bm a\in A} \bm a^T\bm Z\right],$$ where $\bm Z$ follows the standard $n-$dimensional Gaussian distribution $N_n(\bm 0, I_n)$. Finally, the \textit{gradient-complexity} of $h$ is defined as: $$D(h) = \mathbf{GW}\left(\{\nabla h(\bm y): \bm y \in \{0,1\}^n\} \cup \{\bm 0\}\right),$$ where $\nabla h(\bm y) := (\partial_1h(\bm y),\cdots,\partial_nh(\bm y))^\top$.

For $\bm x \in [0, 1]$, define the function: 
\begin{align}\label{eq:fn_hypergraph_ut}
f_n(\bm x) := \frac{|V(H_n)|}{|E(H_n)|}t(H_n,\bm x),
\end{align}
where $t(H_n,\bm x)$ is as in Definition \ref{defn:approx}. It follows from \cite[Theorem 5]{eldan} that a sequence of  $s$-uniform hypergraph $\{H_n\}_{n \geq 1}$ is mean-field for the upper tail problem if 
\begin{align}\label{eq:ut_grad}
\mathrm{Lip}(f_n)=O(1) \quad \text{and} \quad D(f_n) = o(|V(H_n)|).
\end{align}
In \cite[Corollary 6.1]{Briet} it was proved that the above conditions are satisfied whenever \eqref{eq:suniform} holds. The first condition in \eqref{eq:suniform} ensures $D(f_n) = o(|V(H_n)|)$, while the second condition implies $\mathrm{Lip}(f_n)=O(1)$. In the following proposition, we derive another easy sufficient condition for \eqref{eq:ut_grad} to hold, in terms of the number of hyperdges in $H_n$. As a special case, this recovers the mean-field condition for graphs \eqref{eq:graph}.

\begin{ppn}\label{ppn:ut_hypergraph} A sequence of $s$-uniform hypergraphs $\{H_n\}_{n \geq 1}$ is mean-field for the upper tail problem if 
\begin{align}\label{eq:hypergraph}
|E(H_n)| \gg |V(H_n)|^{s-1} \quad \text{and} \quad  \Delta_1(H_n)=O\left(\frac{|E(H_n)|}{|V(H_n)|}\right),
\end{align}
where $\Delta_1(H_n)$ denotes the maximum degree of the hypergraph $H_n$. 
\end{ppn}

\begin{proof} Recall the definition of $f_n(\cdot)$ from \eqref{eq:fn_hypergraph_ut}. Note that,
\begin{equation*}
\mathrm{Lip}(f_n) \leq \frac{|V(H_n)|}{|E(H_n)|} \max_{v\in V(H_n)}\big|\{e \in E(H_n) :~v \in e\}\big| = \frac{|V(H_n)|}{|E(H_n)|}\Delta_1(H_n).
\end{equation*}
The second condition in \eqref{eq:hypergraph} then implies that $\mathrm{Lip}(f_n) = O(1)$.

Therefore, it suffices to show $D(f_n) = o(|V(H_n)|)$ (recall \eqref{eq:ut_grad}). Note that for a standard $|V(H_n)|-$dimensional Gaussian vector $\bm Z$ and $\bm y \in \{0,1\}^{|V(H_n)|}$, 
 \begin{eqnarray*}
 	|\nabla f_n(\bm y)^T \bm Z| &=& \Bigg|\sum_{v \in V(H_n)} \partial_v f_n(\bm y) Z_v\Bigg|\\&=& \frac{|V(H_n)|}{|E(H_n)|} \Bigg|\sum_{v \in V(H_n)} Z_v\sum_{e \in E(H_n) : v \in e}~ \prod_{u \in e\setminus\{v\}} y_u \Bigg|\\&=& \frac{|V(H_n)|}{|E(H_n)|} \Bigg|\sum_{v \in V(H_n)} Z_v\sum_{A \subseteq V(H_n)\setminus\{v\}}\bm 1\left\{A\bigcup \{v\}\in E(H_n)\right\} \prod_{u \in A} y_u\Bigg|\\&=& \frac{|V(H_n)|}{|E(H_n)|} \Bigg|\sum_{A\subseteq V(H_n): |A|=s-1}~ C_A \prod_{u\in A} y_u\Bigg|,
 \end{eqnarray*}
 where $C_A := \sum_{v\notin A} Z_v \cdot  \bm 1\left(A\bigcup \{v\}\in E(H_n)\right)$. By the Cauchy-Schwarz inequality, 
 $$|\nabla f_n(\bm y)^T \bm Z| \leq \frac{|V(H_n)|^{\frac{s+1}{2}}}{|E(H_n)|}\sqrt{\sum_{A \subseteq V(H_n): |A|=s-1} C_A^2} = \frac{|V(H_n)|^{\frac{s+1}{2}}}{|E(H_n)|}\sqrt{\bm C^\top \bm C},$$ where $\bm C$ is the column vector of length $\binom{|V(H_n)|}{s-1}$, having entries $C_A$ indexed by subsets $A$ of $V(H_n)$ having size $s-1$. Note that $\bm C = \mathbf{M} \bm Z$, where $\mathbf{M}=((M_{A, v}))$ is the $\binom{|V(H_n)|}{s-1} \times n$ matrix with entries $M_{A,v} := \bm 1\left\{v\notin A, A\bigcup \{v\} \in E(H_n)\right\}$, where $A$ varies over the set of all subsets of $V(H_n)$ having size $s-1$, and $v \in V(H_n)$. Hence, 
 \begin{eqnarray*}
 	D(f_n) &\leq& \frac{|V(H_n)|^{\frac{s+1}{2}}}{|E(H_n)|} \mathbb{E}\left(\sqrt{\bm Z^\top \mathbf{M}^\top \mathbf{M}\bm Z}\right)\nonumber\\ &\leq&  \frac{|V(H_n)|^{\frac{s+1}{2}}}{|E(H_n)|} \sqrt{\mathbb{E}\left(\bm Z^\top \mathbf{M}^\top \mathbf{M}\bm Z\right)}    \nonumber\\&=&\frac{|V(H_n)|^{\frac{s+1}{2}}}{|E(H_n)|}\sqrt{\mathrm{trace}\left[\mathbf{M}^\top \mathbf{M}\right]}\nonumber\\       &=&\frac{|V(H_n)|^{\frac{s+1}{2}}}{|E(H_n)|}\sqrt{\sum_{v\in V(H_n)}~\sum_{A\subseteq V(H_n): |A|=s-1}\bm 1\left\{v \notin A, A\bigcup\{v\}\in E(H_n)\right\}}\nonumber\\&=& \frac{|V(H_n)|^{\frac{s+1}{2}}}{|E(H_n)|}\sqrt{\sum_{v\in V(H_n)}d_{H_n}(v)} = \frac{\sqrt{s}|V(H_n)|^{\frac{s+1}{2}}}{\sqrt{|E(H_n)|}}= o(|V(H_n)|),
 \end{eqnarray*}
by the first condition in \eqref{eq:hypergraph}.
\end{proof}

In general, the conditions in Proposition \ref{ppn:ut_hypergraph} and those in \eqref{eq:suniform} are incomparable. The maximum co-degree condition in  \eqref{eq:suniform} holds for many sparse  hypergraphs, such as the $r$-AP counting hypergraph. On the other hand, Proposition \ref{ppn:ut_hypergraph} requires very high edge-density, but allows for larger co-degrees, as illustrated in the example below. To this end, for a finite set $T$ and a positive integer $s \geq 1$, denote by ${T \choose s}$ the set of all subsets of $T$ with cardinality $s$. 

\begin{example}\label{example:edge_condition} Fix $s \geq 3$, and take $(s-1)^{-2} < a< (s-1)^{-1}$. For $j \in \{1,2,\ldots, \ceil{n^a}\}$, suppose $$T_j := \{(j-1) \ceil{n^{1-a}} + 1, (j-1) \ceil{n^{1-a}} + 2, \ldots, j \ceil{n^{1-a}}\},$$ and let $F_{n}^{(j)}$ be the hypergraph with vertex set $T_j$ and hyperedge set ${T_j \choose s}$. Note that $$F_{n}^{(1)}, F_{n}^{(2)}, \ldots, F_{n}^{(\ceil{n^a})},$$ is a collection of  $\ceil{n^a}$ disjoint copies of the  complete $s$-uniform hypergraph on $\ceil{n^{1-a}}$ vertices.  Next, define the hypergraph $\cD_n=(V(\cD_n), E(\cD_n))$ with vertex set $V(\cD_n):=\{1', 2', \ldots, n'\}$ and hyperedge set 
$$E(\cD_n):= \left\{B \bigcup \{1', 2'\}:  B \in \binom{V(\cD_n)\backslash\{1', 2'\}}{s-2}\right\},$$
the collection of all $s$-element subsets of $V(\cD_n)$ containing $1'$ and $2'$. Finally, let $H_n$ be the sequence of hypergraphs obtained by taking the disjoint union of 
$$F_{n}^{(1)}, F_{n}^{(2)}, \ldots, F_{n}^{(\ceil{n^a})} \quad \text{and} \quad \cD_n.$$ 
Note that $|E(H_n)| = \ceil{n^a} |E(F_{n}^{(1)})| + {n-2\choose s-2}= \Theta(n^{s - a(s-1)})$. Moreover, 
$$\Delta_1(H_n) \leq \max\{n^{(1-a)(s-1)}, n^{s-2}\} = O(n^{(1-a)(s-1)}),$$ and $\Delta_2(H_n) \gtrsim n^{s-2}$. Now, it is easy to check that condition \eqref{eq:hypergraph} is satisfied. On the other hand, using $ a> (s-1)^{-2}$, $$\frac{|E(H_n)|}{|V(H_n)|^{2-\frac{1}{2\lceil\frac{s-1}{2}\rceil}}} \lesssim n^{s-2 -a(s-1)+\frac{1}{2\lceil\frac{s-1}{2}\rceil}}  \ll n^{s-2},$$ showing that the first condition in \eqref{eq:suniform} is not satisfied.
\end{example}

For the case $s=2$ (which corresponds to graphs),  \eqref{eq:suniform} always implies \eqref{eq:hypergraph}. To see this note that for any sequence of non-empty graphs $\{G_n\}_{n \geq 1}$, $\Delta_2(G_n)=1$. Therefore, the first condition in \eqref{eq:suniform} is equivalent to $|E(G_n)| \gg 
|V(G_n)|^{\frac{3}{2}} \sqrt{\log |V(G_n)|}$, which implies $|E(G_n)| \gg 
|V(G_n)|$, the first condition in \eqref{eq:hypergraph}.  For an example where \eqref{eq:hypergraph} holds, but \eqref{eq:suniform} does not, consider a sequence $\{G_n\}_{n \geq 1}$ of $d$-regular graphs on $n$ vertices, with $1 \ll d = O(\sqrt{n})$. In this case, condition  \eqref{eq:hypergraph} is  trivially satisfied (note that $|E(G_n)|= \frac{nd}{2} \gg n$), but the first condition in \eqref{eq:suniform} does not hold, since $$|E(G_n)| = \frac{nd}{2} \ll n^{\frac{3}{2}} \sqrt{\log n}.$$

\end{document}